\documentclass[9pt]{article}

\usepackage{
  amsmath,
  amsthm,
  amssymb,
  mathtools,
  bm,
}
\usepackage{graphicx}
\usepackage{
  enumerate,
  color,
  multirow,
}

\usepackage{tikz}
\usetikzlibrary{positioning}
\usetikzlibrary{arrows}

\newtheorem{thm}{Theorem}
\newtheorem{prop}[thm]{Proposition}
\newtheorem{lem}[thm]{Lemma}

\theoremstyle{definition}

\title{Multiplicative partition functions for reverse plane partitions derived from an integrable dynamical system}
\author{Shuhei Kamioka}

\begin{document}
\maketitle

\begin{abstract}
  A close connection of reverse plane partitions with
  an integrable dynamical system called the discrete two-dimensional (2D) Toda molecule is clarified.
  It is shown that a multiplicative partition function for reverse plane partition
  of arbitrary shape with bounded parts
  can be obtained from each non-vanishing solution to the discrete 2D Toda molecule.
  As an example a partition function which generalizes MacMahon's triple product formula as well as
  Gansner's multi-trace generating function is derived from a specific solution to the dynamical system.
\end{abstract}

\section{Introduction}

Let $\lambda$ be an (integer) partition or the corresponding Young diagram.
A {\em reverse plane partition} of shape $\lambda$ is
a filling of cells in $\lambda$ with nonnegative integers such that all rows and columns are weakly increasing.
One of the most prominent results in the study of reverse (or ordinary) plane partitions is
the discovery of {\em multiplicative} generating functions, namely those which can be nicely factored.

The first discovery is due to MacMahon \cite{MacMahon(1916CA2)} who proved the triple product formula
\begin{align} \label{eq:MacMahon}
  \sum_{\pi} q^{|\pi|} = \prod_{i=1}^{r} \prod_{j=1}^{c} \prod_{k=1}^{n} \frac{1-q^{i+j+k-1}}{1-q^{i+j+k-2}}
\end{align}
for plane partitions $\pi$ of $r \times c$ rectangular shape with parts at most $n$.
MacMahon's study on plane partitions was later revived by Stanley \cite{Stanley(1971)}.
Among his vast amounts of results on plane partitions
a multiplicative generating function involving the trace statistic is of great importance \cite{Stanley(1973)}.
Stanley's trace generating function was much refined by Gansner \cite{Gansner(1981HG)} who derived
the multi-trace generating function
\begin{align} \label{eq:Gansner}
  \sum_{\pi \in \mathrm{RPP}(\lambda)} \prod_{\ell=1-r}^{c-1} x_{\ell}^{\mathsf{tr}_{\ell}(\pi)}
  = \prod_{(i,j) \in \lambda} \left( 1 - \prod_{\ell=j-\lambda'_j}^{\lambda_i-i} x_{\ell} \right)^{-1}
\end{align}
where $\mathrm{RPP}(\lambda)$ denotes
the set of reverse plane partitions $\pi = (\pi_{i,j})$ of shape $\lambda$ with $r$ rows and $c$ columns,
$\mathsf{tr}_{\ell}(\pi) = \sum_{-i+j=\ell} \pi_{i,j}$ the $\ell$-trace, and
$\lambda'$ the shape conjugate with $\lambda$.
Note that plane partitions considered in \eqref{eq:MacMahon} are those with bounded parts but
reverse plane partitions in \eqref{eq:Gansner} are those with unbounded parts.

In this paper a close connection of reverse plane partitions with an integrable dynamical system, called
the {\em discrete two-dimensional (2D) Toda molecule} \cite{Hirota-Tsujimoto-Imai(1993RIMS)}, is clarified.
Especially it is shown that a multiplicative partition functions for reverse plane partitions can be derived from
each non-vanishing solution to the dynamical system (Theorem \ref{thm:PFsProd} in Section \ref{sec:partitionFunctions})
where reverse plane partitions considered are those of arbitrary shape with bounded parts.
As a concrete example a partition function which generalizes
both MacMahon's triple product formula \eqref{eq:MacMahon} and
Gansner's multi-trace generating function \eqref{eq:Gansner} is derived from
a specific solution (Theorem \ref{thm:partitionFunctionExample} in Section \ref{sec:example}).
The result is an extension of the author's partition function for
rectangular-shaped plane partitions \cite{Kamioka(2015+),Kamioka(2016)}.

The key idea comes from a combinatorial interpretation of the discrete 2D Toda molecule in terms of
non-intersecting lattice paths (Section \ref{sec:latticePaths}).
Note that Viennot \cite{Viennot(2000)} takes a similar approach to count non-intersecting Dyck paths by using
the quotient-difference (qd) algorithm for Pad\'e approximation.


\section{Solutions to the discrete 2D Toda molecule}
\label{sec:solutions}

We see a brief review on the integrable dynamical system discussed throughout the paper.
The discrete two-dimensional (2D) Toda molecule is one of the most typical integrable dynamical systems
that was introduced as a discrete analogue of the Toda lattice \cite{Hirota-Tsujimoto-Imai(1993RIMS)}.
The evolution of the discrete 2D Toda molecule is described by the difference equations
\begin{subequations} \label{eq:d2DToda}
  \begin{gather}
    \label{eq:d2DTodaSum}
    a^{(s,t+1)}_{n} + b^{(s+1,t)}_{n} = a^{(s,t)}_{n} + b^{(s,t)}_{n+1}, \\
    \label{eq:d2DTodaProd}
    a^{(s,t+1)}_{n} b^{(s+1,t)}_{n+1} = a^{(s,t)}_{n+1} b^{(s,t)}_{n+1}, \\
    \label{eq:d2DTodaBC}
    (s,t) \in \mathbb{Z}^{2}, \quad n \in \mathbb{Z}_{\ge 0}, \quad b^{(s,t)}_{0} = 0.
  \end{gather}
\end{subequations}
Through a dependent variable transformation
\begin{align} \label{eq:depVarTrf}
  a^{(s,t)}_{n} = \frac{\tau^{(s+1,t)}_{n+1} \tau^{(s,t)}_{n}}{\tau^{(s+1,t)}_{n} \tau^{(s,t)}_{n+1}}, \qquad
  b^{(s,t)}_{n} = \frac{\tau^{(s,t+1)}_{n-1} \tau^{(s,t)}_{n+1}}{\tau^{(s,t+1)}_{n} \tau^{(s,t)}_{n}}
\end{align}
with $\tau^{(s,t)}_{0} = 1$ the following equation is obtained from \eqref{eq:d2DToda}
\begin{subequations} \label{eq:bilinearEq}
  \begin{gather}
    \tau^{(s+1,t+1)}_{n-1} \tau^{(s,t)}_{n+1}
    - \tau^{(s+1,t+1)}_{n} \tau^{(s,t)}_{n}
    + \tau^{(s+1,t)}_{n} \tau^{(s,t+1)}_{n} = 0, \\
    (s,t) \in \mathbb{Z}^{2}, \quad n \in \mathbb{Z}_{\ge 1}, \quad \tau^{(s,t)}_{0} = 1,
  \end{gather}
\end{subequations}
that is the so-called bilinear form of the discrete 2D Toda molecule.

The bilinear form \eqref{eq:bilinearEq} is solved by the determinant
\begin{align} \label{eq:detSol}
  \tau^{(s,t)}_{n} = \det_{0 \le i,j < n}(f_{s+i,t+j}) =
  \begin{vmatrix}
    f_{s,t}     & \cdots & f_{s,t+j}     & \cdots & f_{s,t+n-1}     \\
    \vdots      &        & \vdots        &        & \vdots          \\
    f_{s+i,t}   & \cdots & f_{s+i,t+j}   & \cdots & f_{s+i,t+n-1}   \\
    \vdots      &        & \vdots        &        & \vdots          \\
    f_{s+n-1,t} & \cdots & f_{s+n-1,t+j} & \cdots & f_{s+n-1,t+n-1} \\
  \end{vmatrix}
\end{align}
where $f = f_{i,j}$ is an arbitrary function defined over $\mathbb{Z}^{2}$.
The discrete 2D Toda molecule \eqref{eq:d2DToda} is therefore solved by $a^{(s,t)}_{n}$, $b^{(s,t)}_{n}$ given by
\eqref{eq:depVarTrf} with \eqref{eq:detSol} if the determinant does not vanish.
Conversely there exists a function $f$ on $\mathbb{Z}^{2}$ which satisfies
\eqref{eq:depVarTrf} with \eqref{eq:detSol} for
any solution $a^{(s,t)}_{n} \neq 0$, $b^{(s,t)}_{n} \neq 0$ to \eqref{eq:d2DToda}.
It is not difficult to see the following correspondence between $a^{(s,t)}_{n}$, $b^{(s,t)}_{n}$ and $f$.

\begin{prop} \label{prop:solutions}
  For each solution $a^{(s,t)}_{n} \neq 0$, $b^{(s,t)}_{n} \neq 0$ to
  the discrete 2D Toda molecule \eqref{eq:d2DToda}
  there exists a function $f = f_{i,j}$ on $\mathbb{Z}^{2}$ which gives
  the same solution through \eqref{eq:depVarTrf} with \eqref{eq:detSol}.
  Moreover such an $f$ is uniquely determined up to the transformation $f_{i,j} \to \varphi_j f_{i,j}$
  by any non-vanishing function $\varphi = \varphi_j$ on $\mathbb{Z}$.
\end{prop}

Giving a non-vanishing solution $a^{(s,t)}_{n}$, $b^{(s,t)}_{n}$ to
the discrete 2D Toda molecule is thus essentially equivalent to giving a function $f$ over $\mathbb{Z}^{2}$.

\section{Lattice path combinatorics}
\label{sec:latticePaths}

In this paper we adopt a matrix-like coordinate to draw a square lattice $\mathbb{Z}^{2}$ where
the nearest neighbors $(i+1,j)$, $(i-1,j)$, $(i,j+1)$ and $(i,j-1)$ of a lattice point $(i,j)$ are located on
the south, north, east and west of $(i,j)$ respectively.
We call a subset $\mathbb{L}$ of $\mathbb{Z}^{2}$ {\em regular} such that
(i) if $(i,j) \in \mathbb{L}$ then $(i+k,j+k) \in \mathbb{L}$ for all $k \ge 1$;
(ii) if $(i,j) \in \mathbb{L}$ then $(i-k,j) \not\in \mathbb{L}$ and $(i,j-k) \not\in \mathbb{L}$ for some $k \ge 1$.
We call a point $(i,j) \in \mathbb{L}$ a {\em north boundary point} if $(i-1,j) \not\in \mathbb{L}$;
similarly a {\em west boundary point} if $(i,j-1) \not\in \mathbb{L}$.
We call a point $(i,j) \in \mathbb{L}$ a {\em convex corner} if $(i,j)$ is a north and south boundary point.
The interest is in lattice paths on
a regular subset $\mathbb{L}$ of $\mathbb{Z}^{2}$ consisting of north and east steps.
See Figure \ref{fig:regularSubset} for example.
\begin{figure}
  \centering
  \begin{tikzpicture}[x=1em,y=1em,font=\small]
    \draw [gray,->] (0,2) -- (0,-9) node [below] {$x$};
    \draw [gray,->] (-2,0) -- (9,0) node [right] {$y$};
    \foreach \j in {7,...,8} {\fill (\j,2) circle (0.1);}
    \foreach \j in {4,...,8} {\fill (\j,1) circle (0.1);}
    \foreach \j in {4,...,8} {\fill (\j,0) circle (0.1);}
    \foreach \j in {3,...,8} {\fill (\j,-1) circle (0.1);}
    \foreach \j in {1,...,8} {\fill (\j,-2) circle (0.1);}
    \foreach \j in {1,...,8} {\fill (\j,-3) circle (0.1);}
    \foreach \j in {0,...,8} {\fill (\j,-4) circle (0.1);}
    \foreach \j in {0,...,8} {\fill (\j,-5) circle (0.1);}
    \foreach \j in {0,...,8} {\fill (\j,-6) circle (0.1);}
    \foreach \j in {-1,...,8} {\fill (\j,-7) circle (0.1);}
    \foreach \j in {-2,...,8} {\fill (\j,-8) circle (0.1);}
    \draw [thick] (0,-6) -- ++(2,0) -- ++(0,1) -- ++(2,0) -- ++(0,1) -- ++(1,0) -- ++(0,3) -- ++(1,0) -- ++(0,2);
    \draw [red] (7,2) circle (0.25) (4,1) circle (0.25) (4,0) circle (0.25) (3,-1) circle (0.25) (1,-2) circle (0.25) (1,-3) circle (0.25) (0,-4) circle (0.25) (0,-5) circle (0.25) (0,-6) circle (0.25) (-1,-7) circle (0.25) (-2,-8) circle (0.25);
    \draw [blue] (-2,-8) circle (0.4) (-1,-7) circle (0.4) (0,-4) circle (0.4) (1,-2) circle (0.4) (2,-2) circle (0.25) (3,-1) circle (0.4) (4,1) circle (0.4) (5,1) circle (0.25) (6,1) circle (0.25) (7,2) circle (0.4) (8,2) circle (0.25);
  \end{tikzpicture}
  \caption{%
    A regular subset $\mathbb{L}$ of $\mathbb{Z}^{2}$ and a lattice path.
    North and west boundary points are marked blue and red respectively.
    Convex corners are those marked in both the colors.}
  \label{fig:regularSubset}
\end{figure}
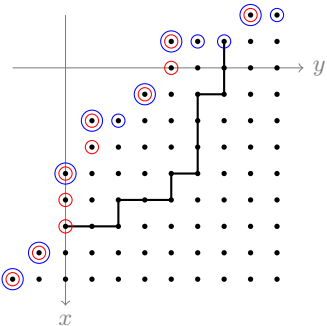

We think of a regular subset $\mathbb{L}$ of $\mathbb{Z}^{2}$ as a graph with
vertices $\mathbb{L}$ and edges connecting nearest neighbors.
We determine the weights of edges by using
a solution $a^{(s,t)}_{n}$, $b^{(s,t)}_{n}$ to the discrete 2D Toda molecule \eqref{eq:d2DToda} as follows.
\begin{enumerate}[(a)]
\item
  The vertical edge with north endpoint at $(i,j)$ is weighted by $a^{(i-n,j-n)}_{n}$ if
  $(i-n,j-n)$ is a west boundary point of $\mathbb{L}$.
\item
  The vertical edge with south endpoint at $(i,j)$ is weighted by $b^{(i-n,j-n)}_{n}$ if
  $(i-n,j-n)$ is a north boundary point of $\mathbb{L}$.
\item
  Every horizontal edge is weighted by $1$.
\end{enumerate}
See Figure \ref{fig:weights} for example.
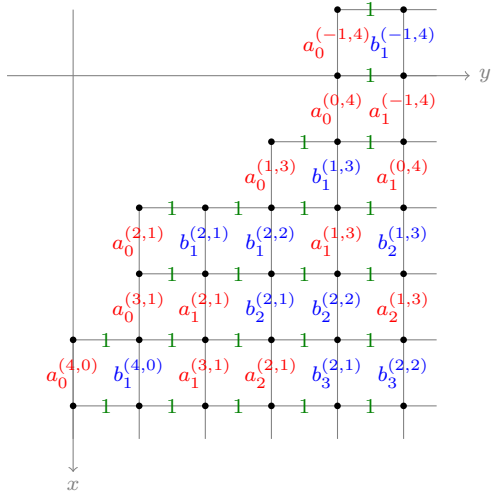
\begin{figure}
  \centering
  \begin{tikzpicture}[x=2.5em,y=2.5em,font=\footnotesize]
    \draw [gray,->] (0,1) -- (0,-6) node [below] {$x$};
    \draw [gray,->] (-1,0) -- (6,0) node [right] {$y$};
    \draw [gray] (4,1) -- (5.5,1) (4,0) -- (5.5,0) (3,-1) -- (5.5,-1) (1,-2) -- (5.5,-2) (1,-3) -- (5.5,-3) (0,-4) -- (5.5,-4) (0,-5) -- (5.5,-5) (0,-4) -- (0,-5.5) (1,-2) -- (1,-5.5) (2,-2) -- (2,-5.5) (3,-1) -- (3,-5.5) (4,1) -- (4,-5.5) (5,1) -- (5,-5.5);
    \foreach \j in {4,...,5} {\fill (\j,1) circle (0.05);}
    \foreach \j in {4,...,5} {\fill (\j,0) circle (0.05);}
    \foreach \j in {3,...,5} {\fill (\j,-1) circle (0.05);}
    \foreach \j in {1,...,5} {\fill (\j,-2) circle (0.05);}
    \foreach \j in {1,...,5} {\fill (\j,-3) circle (0.05);}
    \foreach \j in {0,...,5} {\fill (\j,-4) circle (0.05);}
    \foreach \j in {0,...,5} {\fill (\j,-5) circle (0.05);}
    \foreach \k in {0,...,0} {\node [red] at (0+\k,-4-1/2-\k) {$a^{(4,0)}_{\k}$};}
    \foreach \k in {1,...,1} {\node [blue] at (0+\k,-4+1/2-\k) {$b^{(4,0)}_{\k}$};}
    \foreach \k in {0,...,1} {\node [red] at (1+\k,-3-1/2-\k) {$a^{(3,1)}_{\k}$};}
    \foreach \k in {0,...,2} {\node [red] at (1+\k,-2-1/2-\k) {$a^{(2,1)}_{\k}$};}
    \foreach \k in {1,...,3} {\node [blue] at (1+\k,-2+1/2-\k) {$b^{(2,1)}_{\k}$};}
    \foreach \k in {1,...,3} {\node [blue] at (2+\k,-2+1/2-\k) {$b^{(2,2)}_{\k}$};}
    \foreach \k in {0,...,2} {\node [red] at (3+\k,-1-1/2-\k) {$a^{(1,3)}_{\k}$};}
    \foreach \k in {1,...,2} {\node [blue] at (3+\k,-1+1/2-\k) {$b^{(1,3)}_{\k}$};}
    \foreach \k in {0,...,1} {\node [red] at (4+\k,-0-1/2-\k) {$a^{(0,4)}_{\k}$};}
    \foreach \k in {0,...,1} {\node [red] at (4+\k,1-1/2-\k) {$a^{(-1,4)}_{\k}$};}
    \foreach \k in {1,...,1} {\node [blue] at (4+\k,1+1/2-\k) {$b^{(-1,4)}_{\k}$};}
    \foreach \j in {4,...,4} {\node [green!50!black] at (\j+1/2,1) {$1$};}
    \foreach \j in {4,...,4} {\node [green!50!black] at (\j+1/2,0) {$1$};}
    \foreach \j in {3,...,4} {\node [green!50!black] at (\j+1/2,-1) {$1$};}
    \foreach \j in {1,...,4} {\node [green!50!black] at (\j+1/2,-2) {$1$};}
    \foreach \j in {1,...,4} {\node [green!50!black] at (\j+1/2,-3) {$1$};}
    \foreach \j in {0,...,4} {\node [green!50!black] at (\j+1/2,-4) {$1$};}
    \foreach \j in {0,...,4} {\node [green!50!black] at (\j+1/2,-5) {$1$};}
  \end{tikzpicture}
  \caption{The weights of edges.}
  \label{fig:weights}
\end{figure}
We define the weight $w(\mathbb{L}; a,b; P)$ of a lattice path $P$ on $\mathbb{L}$ to be
the product of the weights of all edges passed by $P$.
We conventionally consider empty paths $P$ with no steps for which $w(\mathbb{L};a,b;P) = 1$.
For $(i,j) \in \mathbb{L}$ and $(k,\ell) \in \mathbb{L}$ we further define
\begin{align}
  g(\mathbb{L};a,b;i,j;k,\ell) = \sum_{P} w(\mathbb{L};a,b;P)
\end{align}
where the sum ranges over all lattice paths on $\mathbb{L}$ going from $(i,j)$ to $(k,\ell)$.

Let $\mathsf{x}(j)$ denote the $x$-coordinate (or the vertical ---) of
the north boundary point of $\mathbb{L}$ with $y$-coordinate (or horizontal ---) equal to $j$;
let $\mathsf{y}(i)$ the $y$-coordinate of the west boundary point of $\mathbb{L}$ with $x$-coordinate equal to $i$.
The following theorem gives a combinatorial interpretation of the discrete 2D Toda molecule and
refines Proposition \ref{prop:solutions}.

\begin{thm} \label{thm:fundamentalThm}
  Let $a^{(s,t)}_{n} \neq 0$, $b^{(s,t)}_{n} \neq 0$ be a solution to the discrete 2D Toda molecule \eqref{eq:d2DToda},
  and let a function $f = f_{i,j}$ on $\mathbb{Z}^{2}$ give
  the same solution through \eqref{eq:depVarTrf} with \eqref{eq:detSol}.
  Let $\mathbb{L}$ be a regular subset of $\mathbb{Z}^{2}$.
  For $(i,j) \in \mathbb{L}$ then
  \begin{align}
    \frac{f_{i,j}}{f_{\mathsf{x}(j),j}} = g(\mathbb{L};a,b;i,\mathsf{y}(i);\mathsf{x}(j),j).
  \end{align}
\end{thm}

In order to prove the theorem we use the following lemma.

\begin{lem} \label{lem:weightConserved}
  Let $\mathbb{L}$ be a regular subset of $\mathbb{Z}^{2}$ with a convex corner $(s,t) \in \mathbb{L}$.
  Let $\mathbb{L}'$ denote the regular subset of $\mathbb{Z}^{2}$ obtained from $\mathbb{L}$ by deleting $(s,t)$.
  For $(i,j)$ and $(k,\ell)$ in $\mathbb{L}'$ with $i-j \neq s-t$ and $k-\ell \neq s-t$ then
  \begin{align} \label{eq:weightConserved}
    g(\mathbb{L};a,b;i,j;k,\ell) = g(\mathbb{L}';a,b;i,j;k,\ell).
  \end{align}
\end{lem}

\begin{proof}
  The difference between $\mathbb{L}$ and $\mathbb{L}'$ is only in
  the existence and the absence of the convex corner $(s,t)$, and
  the weights of vertical edges between the two diagonal lines $d_{-}: y-x = t-s-1$ and $d_{+}: y-x = t-s+1$.
  (The vertical edges between $d_{-}$ and $d_{+}$ are weighted by
  $a^{(s,t)}_{n}$, $b^{(s,t)}_{n}$ on $\mathbb{L}$ and by $a^{(s,t+1)}_{n}$, $b^{(s+1,t)}_{n}$ on $\mathbb{L}'$.)
  Assume that $i-j \neq s-t$ and $k-\ell \neq s-t$ meaning that
  $(i,j)$ and $(k,\ell)$ is outside the region between $d_{-}$ and $d_{+}$. (Those may be on $d_{\pm}$.)
  If both $(i,j)$ and $(k,\ell)$ are either in the south of $d_{-}$ or in the north of $d_{+}$ then
  the identity \eqref{eq:weightConserved} clearly holds since
  lattice paths going from $(i,j)$ to $(k,\ell)$ never enter the region between $d_{-}$ and $d_{+}$.
  In the rest of the proof we thus assume that
  $(i,j)$ is in the south of $d_{-}$ and $(k,\ell)$ in the north of $d_{+}$.

  Each lattice path $P$ going from $(i,j)$ to $(k,\ell)$ is uniquely divided into three subpaths:
  $P_-$ from $(i,j)$ to $d_-$,
  $Q$ of two steps between $d_{-}$ and $d_{+}$ and
  $P_+$ from $d_+$ to $(k,\ell)$.
  Obviously $w(P_{\pm}) = w'(P_{\pm})$ where $w$ and $w'$ are abbreviations of
  $w(\mathbb{L};a,b;\cdot)$ and $w(\mathbb{L}';a,b;\cdot)$ respectively.
  The proof of \eqref{eq:weightConserved} thus amounts to showing that
  $g(i,j;k,\ell) = g'(i,j;k,\ell)$ for each $(i,j)$ on $d_{-}$ and $(k,\ell)$ on $d_{-}$ where
  $g$ and $g'$ are abbreviations of $g(\mathbb{L};a,b;\cdot)$ and $g(\mathbb{L}';a,b;\cdot)$ respectively.
  Since $Q$ is of two steps we have only three cases:
  (i) $(i,j) = (s+n,t+n-1)$ and $(k,\ell) = (s+n,t+n+1)$ for some $n \ge 1$;
  (ii) $(i,j) = (s+n+1,t+n)$ and $(k,\ell) = (s+n,t+n+1)$ for some $n \ge 0$;
  (iii) $(i,j) = (s+n+1,t+n)$ and $(k,\ell) = (s+n-1,t+n)$ for some $n \ge 1$.
  See Figure \ref{fig:lemma}.
  \begin{figure}
    \centering
    \begin{tikzpicture}[x=2.5em,y=2.5em,font=\small]
      \begin{scope}
        \node at (0,1.25) {$\mathbb{L}$ and $\mathbb{L}'$};
        \foreach \j in {-1,0,1} {\fill (\j,0) circle (0.05);}
        \draw [gray,densely dashed] (-1.5,0.5) -- ++(1,-1) node [below right] {$d_-$} (0.5,0.5) -- ++(1,-1) node [below right] {$d_+$};
        \draw [gray] (-1,0) -- node [green!50!black] {$1$} (0,0) -- node [green!50!black] {$1$} (1,0);
        \node at (0,-2) {Case (i)};
      \end{scope}
      \draw [densely dotted] (2.5,2.5) -- ++(0,-5);
      \begin{scope}[xshift=8.5em,yshift=1.25em]
        \begin{scope}
          \node at (0.5,1) {$\mathbb{L}$};
          \foreach \i in {-1,0} {\foreach \j in {0,1} {\fill (\j,\i) circle (0.05);}}
          \draw [gray,dashed] (-0.5,-0.5) -- ++(1,-1) (0.5,0.5) -- ++(1,-1);
          \draw [gray]
          (0,-1) -- node [red] {$a^{(s,t)}_{n}$} (0,0) -- node [green!50!black] {$1$} (1,0)
          (0,-1) -- node [green!50!black] {$1$} (1,-1) -- node [blue] {$b^{(s,t)}_{n+1}$} (1,0);
        \end{scope}
        \begin{scope}[xshift=6.25em]
          \node at (0.5,1) {$\mathbb{L}'$};
          \foreach \i in {-1,0} {\foreach \j in {0,1} {\fill (\j,\i) circle (0.05);}}
          \draw [gray,dashed] (-0.5,-0.5) -- ++(1,-1) (0.5,0.5) -- ++(1,-1);
          \draw [gray]
          (0,-1) -- node [blue] {$b^{(s+1,t)}_{n}$} (0,0) -- node [green!50!black] {$1$} (1,0)
          (0,-1) -- node [green!50!black] {$1$} (1,-1) -- node [red] {$a^{(s,t+1)}_{n}$} (1,0);
        \end{scope}
        \node at (1.75,-2.5) {Case (ii)};
      \end{scope}
      \draw [densely dotted] (8,2.5) -- ++(0,-5);
      \begin{scope}[xshift=23em]
        \begin{scope}
          \node at (0,2) {$\mathbb{L}$};
          \foreach \i in {-1,0,1} {\fill (0,\i) circle (0.05);}
          \draw [gray,dashed] (-0.5,-0.5) -- ++(1,-1) (-0.5,1.5) -- ++(1,-1);
          \draw [gray] (0,-1) -- node [red] {$a^{(s,t)}_{n}$} (0,0) -- node [blue] {$b^{(s,t)}_{n}$} (0,1);
        \end{scope}
        \begin{scope}[xshift=4.5em]
          \node at (0,2) {$\mathbb{L}'$};
          \foreach \i in {-1,0,1} {\fill (0,\i) circle (0.05);}
          \draw [gray,dashed] (-0.5,-0.5) -- ++(1,-1) (-0.5,1.5) -- ++(1,-1);
          \draw [gray] (0,-1) -- node [blue] {$b^{(s+1,t)}_{n}$} (0,0) -- node [red] {$a^{(s,t+1)}_{n-1}$} (0,1);
        \end{scope}
        \node at (1,-2) {Case (iii)};
      \end{scope}
    \end{tikzpicture}
    \caption{Proof of Lemma \ref{lem:weightConserved}.}
    \label{fig:lemma}
  \end{figure}
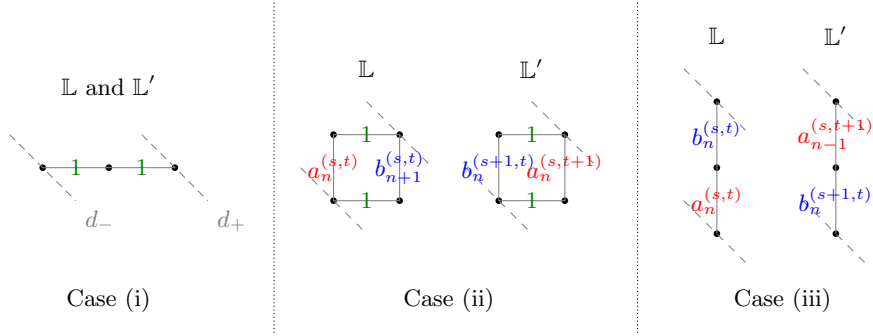

  Case (i):
  The unique lattice path going from $(i,j) = (s+n,t+n-1)$ to $(k,\ell) = (s+n,t+n+1)$ of two east steps is
  both on $\mathbb{L}$ and on $\mathbb{L}'$.
  Thus $g(i,j;k,\ell) = g'(i,j;k,\ell) = 1$.

  Case (ii):
  There are two lattice paths going from $(i,j) = (s+n+1,t+n)$ to $(k,\ell) = (s+n,t+n+1)$
  one of which is $Q_1$ going north and east, the other is $Q_2$ going east and north.
  If $n \ge 1$ then $Q_1$ and $Q_2$ are both on $\mathbb{L}$ and on $\mathbb{L}'$, and
  $w(Q_1) = a^{(s,t)}_{n}$, $w(Q_2) = b^{(s,t)}_{n+1}$, $w'(Q_1) = a^{(s,t+1)}_{n}$ and $w'(Q_2) = b^{(s+1,t)}_{n}$.
  Thus $g(i,j;k,\ell) = a^{(s,t)}_{n} + b^{(s,t)}_{n+1} = a^{(s,t+1)}_{n} + b^{(s+1,t)}_{n} = g'(i,j;k,\ell)$
  because of \eqref{eq:d2DTodaProd}.
  If $n = 0$ then $Q_1$ and $Q_2$ are on $\mathbb{L}$ while only $Q_2$ on $\mathbb{L}'$, and
  $w(Q_1) = a^{(s,t)}_{0}$, $w(Q_2) = b^{(s,t)}_{1}$ and $w'(Q_1) = a^{(s,t+1)}_{0}$.
  ($Q_1$ is not on $\mathbb{L}'$ since $Q_1$ passes through $(s,t) \not\in \mathbb{L}'$.)
  Thus $g(i,j;k,\ell) = a^{(s,t)}_{0} + b^{(s,t)}_{1} = a^{(s,t+1)}_{0} = g'(i,j;k,\ell)$
  because of \eqref{eq:d2DTodaSum} with \eqref{eq:d2DTodaBC}.
  
  Case (iii):
  The unique lattice path going from $(i,j) = (s+n+1,t+n)$ to $(k,\ell) = (s+n-1,t+n)$ of two north steps is
  both on $\mathbb{L}$ and on $\mathbb{L}'$.
  The weight of the lattice path is
  $a^{(s,t)}_{n} b^{(s,t)}_{n}$ on $\mathbb{L}$ and
  $a^{(s,t+1)}_{n-1} b^{(s+1,t)}_{n}$ on $\mathbb{L}'$.
  Thus $g(i,j;k,\ell) = a^{(s,t)}_{n} b^{(s,t)}_{n} = a^{(s,t+1)}_{n-1} b^{(s+1,t)}_{n} = g'(i,j;k,\ell)$
  because of \eqref{eq:d2DTodaProd}.
\end{proof}

\begin{proof}[Proof of Theorem \ref{thm:fundamentalThm}]
  Let $\mathbb{L}'$ denote the regular subset of $\mathbb{Z}^{2}$ defined by
  $\mathbb{L}' = \mathbb{L} \setminus \{ (s,t) \in \mathbb{L};~ \text{$s < i$ and $t < j$} \}$.
  From Lemma \ref{lem:weightConserved} then
  $g(\mathbb{L};a,b;i,\mathsf{y}(i);\mathsf{x}(j),j) = g(\mathbb{L}';a,b;i,\mathsf{y}(i);\mathsf{x}(j),j)$ because
  $\mathbb{L}'$ can be obtained from $\mathbb{L}$ by iterative deletion of convex corners.
  A lattice path going from $(i,\mathsf{y}(i))$ to $(\mathsf{x}(j),j)$ on $\mathbb{L}'$ is unique because
  such a lattice path cannot turn north until $(i,j)$ and cannot turn east from $(i,j)$.
  The weight of the unique lattice path on $\mathbb{L}'$ implies that
  $g(\mathbb{L}';a,b;i,\mathsf{y}(i);\mathsf{x}(j),j) = \prod_{k=\mathsf{x}(j)}^{i-1} a^{(k,j)}_{0}$.
  The last product is equal to $f_{i,j}/f_{\mathsf{x}(j),j}$ because $a^{(k,j)}_{0} = f_{k+1,j} / f_{k,j}$ from
  \eqref{eq:depVarTrf} and \eqref{eq:detSol}.
\end{proof}

Theorem \ref{thm:fundamentalThm} admits
a combinatorial interpretation of the determinant $\tau^{(s,t)}_{n}$ by means of
Gessel--Viennot--Lindstr\"om's method \cite{Gessel-Viennot(1985),Lindstroem(1973)}.
For $(s,t) \in \mathbb{L}$ and $n \ge 0$ we define $\mathrm{LP}(\mathbb{L},s,t,n)$ to be the set of
$n$-tuples $(P_0,\dots,P_{n-1})$ of lattice paths on $\mathbb{L}$ such that
(i) $P_k$ goes from $(s+k,\mathsf{y}(s)+k)$ to $(\mathsf{x}(t)+k,t+k)$ for each $0 \le k < n$, and
(ii) $P_0,\dots,P_{n-1}$ are {\em non-intersecting}: $P_j \cap P_k = \emptyset$ if $j \neq k$.
See Figure \ref{fig:NILatticePaths}, the second figure shows an $n$-tuple of non-intersecting lattice paths.

\begin{thm} \label{thm:NILatticePaths}
  Let $a^{(s,t)}_{n} \neq 0$, $b^{(s,t)}_{n} \neq 0$ be a solution to
  the discrete 2D Toda molecule \eqref{eq:d2DToda}, and
  let a function $f = f_{i,j}$ on $\mathbb{Z}^{2}$ give
  the same solution through \eqref{eq:depVarTrf} with \eqref{eq:detSol}.
  Let $\mathbb{L}$ be a regular subset of $\mathbb{Z}^{2}$.
  For $(s,t) \in \mathbb{L}$ and $n \ge 0$ then
  \begin{align} \label{eq:NILatticePaths}
    \frac{\tau^{(s,t)}_{n}}{\tau^{(\mathsf{x}(t),t)}_{n}}
    = \sum_{(P_0,\dots,P_{n-1}) \in \mathrm{LP}(\mathbb{L},s,t,n)} \prod_{k=0}^{n-1} w(\mathbb{L};a,b;P_k)
  \end{align}
  where $\tau^{(s,t)}_{n} = \det_{0 \le i,j < n}(f_{s+i,t+j})$.
\end{thm}

\begin{proof}
  Gessel--Viennot--Lindstr\"om's method yields from Theorem \ref{thm:fundamentalThm} that
  \begin{align} \label{eq:NILPaths'}
    \frac{\tau^{(s,t)}_{n}}{\prod_{k=0}^{n-1} f_{\mathsf{x}(t+k),t+k}}
    = \sum_{(P'_0,\dots,P'_{n-1})} \prod_{k=0}^{n-1} w(\mathbb{L};a,b;P'_k)
  \end{align}
  where the sum ranges over all $n$-tuples $(P'_0,\dots,P'_{n-1})$ of
  non-intersecting lattice paths on $\mathbb{L}$ such that
  $P'_k$ goes from $(s+k,\mathsf{y}(s+k))$ to $(\mathsf{x}(t+k),t+k)$ for each $0 \le k < n$.
  Eliminating the steps frozen due to the non-intersecting condition
  we obtain $(P_0,\dots,P_{n-1}) \in \mathrm{LP}(\mathbb{L},s,t,n)$,
  see Figure \ref{fig:NILatticePaths} for example.
  The weight of the eliminated frozen steps is equal to
  the weight of a unique configuration of non-intersecting lattice paths from
  $\tau^{(\mathsf{x}(t),t)}_{n} / \prod_{k=0}^{n-1} f_{\mathsf{x}(t+k),t+k}$,
  see the first and the last figures in Figure \ref{fig:NILatticePaths} for example.
  Thus $\tau^{(s,t)}_{n}$ is equal to
  the right-hand side of \eqref{eq:NILatticePaths} multiplied by $\tau^{(\mathsf{x}(t),t)}_{n}$.
  \begin{figure}
    \centering
    \begin{tikzpicture}[x=0.75em,y=0.75em,font=\scriptsize]
      \begin{scope}
        \draw [gray,->] (0,2) -- (0,-9);
        \draw [gray,->] (-2,0) -- (9,0);
        \foreach \j in {7,...,8} {\fill (\j,2) circle (0.1);}
        \foreach \j in {4,...,8} {\fill (\j,1) circle (0.1);}
        \foreach \j in {4,...,8} {\fill (\j,0) circle (0.1);}
        \foreach \j in {3,...,8} {\fill (\j,-1) circle (0.1);}
        \foreach \j in {1,...,8} {\fill (\j,-2) circle (0.1);}
        \foreach \j in {1,...,8} {\fill (\j,-3) circle (0.1);}
        \foreach \j in {0,...,8} {\fill (\j,-4) circle (0.1);}
        \foreach \j in {0,...,8} {\fill (\j,-5) circle (0.1);}
        \foreach \j in {0,...,8} {\fill (\j,-6) circle (0.1);}
        \foreach \j in {-1,...,8} {\fill (\j,-7) circle (0.1);}
        \foreach \j in {-2,...,8} {\fill (\j,-8) circle (0.1);}
        \draw [thick]
        (0,-5) node [left] {$P'_0$} -- ++(0,1) -- ++(2,0) -- ++(0,1) -- ++(1,0) -- ++(0,2) -- ++(1,0) -- ++(0,1) -- ++(1,0) -- ++(0,1)
        (0,-6) node [left] {$P'_1$} -- ++(1,0) -- ++(0,1) -- ++(3,0) -- ++(0,2) -- ++(1,0) -- ++(0,2) -- ++(1,0) -- ++(0,2)
        (-1,-7) node [left] {$P'_2$} -- ++(5,0) -- ++(0,1) -- ++(1,0) -- ++(0,1) -- ++(1,0) -- ++(0,2) -- ++(1,0) -- ++(0,5)
        (-2,-8) node [left] {$P'_3$} -- ++(9,0) -- ++(0,2) -- ++(1,0) -- ++(0,8);
        \draw [very thick,orange]
        (0,-6) -- ++(1,0) (6,1) -- ++(0,-1)
        (-1,-7) -- ++(3,0) (7,2) -- ++(0,-3)
        (-2,-8) -- ++(5,0) (8,2) -- ++(0,-4);
        \node [font=\small] at (3.5,-10.5) {$\tau^{(s,t)}_{n} / \prod_{k=0}^{n-1} f_{\mathsf{x}(t+k),t+k}$};
      \end{scope}
      \draw [ultra thick,->] (10,-3.5) -- ++(1.5,0);
      \begin{scope}[xshift=11em]
        \draw [gray,->] (0,2) -- (0,-9);
        \draw [gray,->] (-2,0) -- (9,0);
        \foreach \j in {7,...,8} {\fill (\j,2) circle (0.1);}
        \foreach \j in {4,...,8} {\fill (\j,1) circle (0.1);}
        \foreach \j in {4,...,8} {\fill (\j,0) circle (0.1);}
        \foreach \j in {3,...,8} {\fill (\j,-1) circle (0.1);}
        \foreach \j in {1,...,8} {\fill (\j,-2) circle (0.1);}
        \foreach \j in {1,...,8} {\fill (\j,-3) circle (0.1);}
        \foreach \j in {0,...,8} {\fill (\j,-4) circle (0.1);}
        \foreach \j in {0,...,8} {\fill (\j,-5) circle (0.1);}
        \foreach \j in {0,...,8} {\fill (\j,-6) circle (0.1);}
        \foreach \j in {-1,...,8} {\fill (\j,-7) circle (0.1);}
        \foreach \j in {-2,...,8} {\fill (\j,-8) circle (0.1);}
        \draw [thick]
        (0,-5) node [left] {$P_0$} -- ++(0,1) -- ++(2,0) -- ++(0,1) -- ++(1,0) -- ++(0,2) -- ++(1,0) -- ++(0,1) -- ++(1,0) -- ++(0,1)
        (1,-6) node [left] {$P_1$} -- ++(0,1) -- ++(3,0) -- ++(0,2) -- ++(1,0) -- ++(0,2) -- ++(1,0) -- ++(0,1)
        (2,-7) node [left] {$P_2$} -- ++(2,0) -- ++(0,1) -- ++(1,0) -- ++(0,1) -- ++(1,0) -- ++(0,2) -- ++(1,0) -- ++(0,2)
        (3,-8) node [left] {$P_3$} -- ++(4,0) -- ++(0,2) -- ++(1,0) -- ++(0,4);
        \node [font=\small] at (3.5,-10.5) {$\mathrm{LP}(\mathbb{L},s,t,n)$};
      \end{scope}
      \draw [densely dotted] (26,3) -- ++(0,-15);
      \begin{scope}[xshift=23em]
        \draw [gray,->] (0,2) -- (0,-9);
        \draw [gray,->] (-2,0) -- (9,0);
        \foreach \j in {7,...,8} {\fill (\j,2) circle (0.1);}
        \foreach \j in {4,...,8} {\fill (\j,1) circle (0.1);}
        \foreach \j in {4,...,8} {\fill (\j,0) circle (0.1);}
        \foreach \j in {3,...,8} {\fill (\j,-1) circle (0.1);}
        \foreach \j in {1,...,8} {\fill (\j,-2) circle (0.1);}
        \foreach \j in {1,...,8} {\fill (\j,-3) circle (0.1);}
        \foreach \j in {0,...,8} {\fill (\j,-4) circle (0.1);}
        \foreach \j in {0,...,8} {\fill (\j,-5) circle (0.1);}
        \foreach \j in {0,...,8} {\fill (\j,-6) circle (0.1);}
        \foreach \j in {-1,...,8} {\fill (\j,-7) circle (0.1);}
        \foreach \j in {-2,...,8} {\fill (\j,-8) circle (0.1);}
        \draw [thick]
        (4,1) -- ++(1,0)
        (4,0) -- ++(2,0) -- ++(0,1)
        (3,-1) -- ++(4,0) -- ++(0,3)
        (1,-2) -- ++(7,0) -- ++(0,4);
        \node [font=\small] at (3.5,-10.5) {$\tau^{(\mathsf{x}(t),t)}_{n} / \prod_{k=0}^{n-1} f_{\mathsf{x}(t+k),t+k}$};
      \end{scope}
    \end{tikzpicture}
    \caption{%
      Proof of Theorem \ref{thm:NILatticePaths} where $(s,t) = (5,4)$ and $n = 4$.
      In the first figure the steps in orange are frozen due to the non-intersecting condition.}
    \label{fig:NILatticePaths}
  \end{figure}
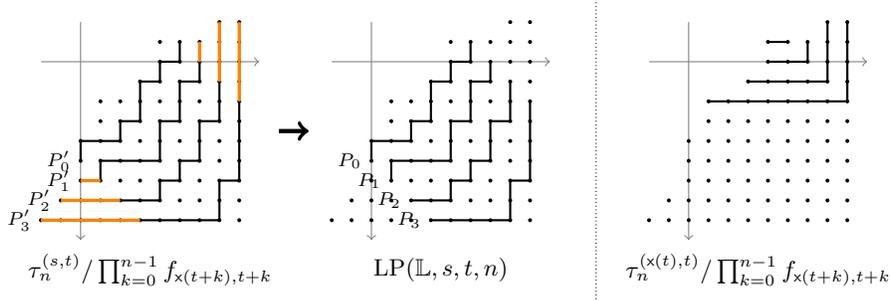
\end{proof}

Note that the left-hand side of \eqref{eq:NILatticePaths} can be expressed as
\begin{align} \label{eq:tauProd}
  \frac{\tau^{(s,t)}_{n}}{\tau^{(\mathsf{x}(t),t)}_{n}}
  = \prod_{i=1}^{s-\mathsf{x}(t)} \prod_{k=1}^{n} a^{(s-i,t)}_{k-1}
\end{align}
from \eqref{eq:depVarTrf}.
We can readily evaluate the sum in \eqref{eq:NILatticePaths}, a partition function for non-intersecting lattice paths,
by using this formula.

\section{Multiplicative partition functions for reverse plane partitions}
\label{sec:partitionFunctions}

Let $\lambda$ be a partition and let $n \ge 0$.
We write $\mathrm{RPP}(\lambda,n)$ for
the set of reverse plane partitions of shape $\lambda$ with parts at most $n$.
Let $r$ and $c$ denote the numbers of rows and columns in $\lambda$ respectively.
We then define a regular subset $\mathbb{L}(\lambda)$ of $\mathbb{Z}^{2}$ by
\begin{align}
  \mathbb{L}(\lambda) = \{ (i,j) \in \mathbb{Z}_{\ge 0}^{2};~ j \ge c - \lambda_{r-i} \}
\end{align}
where $\lambda_i$ denotes the $i$-th part of $\lambda$ for $1 \le i \le r$ and $\lambda_i = c$ for $i \le 0$.
There is a bijection between $\mathrm{LP}(\mathbb{L}(\lambda),r,c,n)$ and $\mathrm{RPP}(\lambda,n)$
which is described as follows.
Given an $n$-tuple $(P_0,\dots,P_{n-1}) \in \mathrm{LP}(\mathbb{L}(\lambda),r,c,n)$ of
non-intersecting lattice paths on $\mathbb{L}(\lambda)$
\begin{enumerate}[(i)]
\item
  move the lattice path $P_k$ northwest by $(-k,-k)$ for each $0 \le k < n$;
\item
  fill in the cells between $P_{n-k-1}$ and $P_{n-k}$ with $k$ for each $0 \le k \le n$ where
  $P_{-1}$ is the lattice path going from $(r,0)$ to $(0,c)$ along the border of $\mathbb{L}(\lambda)$ and
  $P_{n}$ is that going east from $(r,0)$, turning north at $(r,c)$ and going north to $(0,c)$;
\item
  rotate 180$^{\circ}$ to obtain a reverse plane partition in $\mathrm{RPP}(\lambda,n)$.
\end{enumerate}
Figure \ref{fig:bijection} demonstrates the bijection by an example.
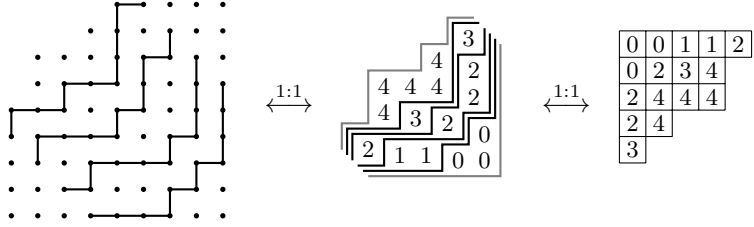
\begin{figure}
  \centering
  \begin{tikzpicture}[x=1em,y=1em]
    \begin{scope}
      \foreach \j in {4,...,8} {\fill (\j,0) circle (0.1);}
      \foreach \j in {3,...,8} {\fill (\j,-1) circle (0.1);}
      \foreach \j in {1,...,8} {\fill (\j,-2) circle (0.1);}
      \foreach \j in {1,...,8} {\fill (\j,-3) circle (0.1);}
      \foreach \i in {4,...,8} {\foreach \j in {0,...,8} {\fill (\j,-\i) circle (0.1);}}
      \draw [thick] (0,-5) -- ++(0,1) -- ++(2,0) -- ++(0,1) -- ++(2,0) -- ++(0,3) -- ++(1,0);
      \draw [thick] (1,-6) -- ++(0,1) -- ++(3,0) -- ++(0,1) -- ++(1,0) -- ++(0,2) -- ++(1,0) -- ++(0,1);
      \draw [thick] (2,-7) -- ++(1,0) -- ++(0,1) -- ++(3,0) -- ++(0,1) -- ++(1,0) -- ++(0,3);
      \draw [thick] (3,-8) -- ++(3,0) -- ++(0,1) -- ++(1,0) -- ++(0,1) -- ++(1,0) -- ++(0,3);
    \end{scope}
    \node at (10.5,-3.5) {$\overset{1:1}{\longleftrightarrow}$};
    \begin{scope}[xshift=13em,yshift=-1em]
      \draw [thick,gray] (0,-5) ++(-0.5,0.5) -- ++(0,1) -- ++(1,0) -- ++(0,2) -- ++(2,0) -- ++(0,1) -- ++(1,0) -- ++(0,1) -- ++(1,0);
      \draw [thick] (0,-5) ++(-0.3,0.3) -- ++(0,1) -- ++(2,0) -- ++(0,1) -- ++(2,0) -- ++(0,3) -- ++(1,0);
      \draw [thick] (0,-5) ++(-0.1,0.1) -- ++(0,1) -- ++(3,0) -- ++(0,1) -- ++(1,0) -- ++(0,2) -- ++(1,0) -- ++(0,1);
      \draw [thick] (0,-5) ++(0.1,-0.1) -- ++(1,0) -- ++(0,1) -- ++(3,0) -- ++(0,1) -- ++(1,0) -- ++(0,3);
      \draw [thick] (0,-5) ++(0.3,-0.3) -- ++(3,0) -- ++(0,1) -- ++(1,0) -- ++(0,1) -- ++(1,0) -- ++(0,3);
      \draw [thick,gray] (0,-5) ++(0.5,-0.5) -- ++(5,0) -- ++(0,5);
      \begin{scope}[font=\small]
        \node at (5-1/2+0.4,-5+1/2-0.4) {0};
        \node at (4-1/2+0.4,-5+1/2-0.4) {0};
        \node at (5-1/2+0.4,-4+1/2-0.4) {0};
        \node at (3-1/2+0.2,-5+1/2-0.2) {1};
        \node at (2-1/2+0.2,-5+1/2-0.2) {1};
        \node at (1-1/2+0.0,-5+1/2-0.0) {2};
        \node at (4-1/2+0.0,-4+1/2-0.0) {2};
        \node at (5-1/2+0.0,-3+1/2-0.0) {2};
        \node at (5-1/2+0.0,-2+1/2-0.0) {2};
        \node at (3-1/2-0.2,-4+1/2+0.2) {3};
        \node at (5-1/2-0.2,-1+1/2+0.2) {3};
        \node at (2-1/2-0.4,-4+1/2+0.4) {4};
        \node at (2-1/2-0.4,-3+1/2+0.4) {4};
        \node at (3-1/2-0.4,-3+1/2+0.4) {4};
        \node at (4-1/2-0.4,-3+1/2+0.4) {4};
        \node at (4-1/2-0.4,-2+1/2+0.4) {4};
      \end{scope}
    \end{scope}
    \node at (21,-3.5) {$\overset{1:1}{\longleftrightarrow}$};
    \begin{scope}[xshift=23em,yshift=-1em,font=\small]
      \draw
      (0,0) -- ++(5,0) (0,-1) -- ++(5,0) (0,-2) -- ++(4,0) (0,-3) -- ++(4,0) (0,-4) -- ++(2,0) (0,-5) -- ++(1,0)
      (0,0) -- ++(0,-5) (1,0) -- ++(0,-5) (2,0) -- ++(0,-4) (3,0) -- ++(0,-3) (4,0) -- ++(0,-3) (5,0) -- ++(0,-1);
      \node at (1-1/2,-1+1/2) {0};
      \node at (2-1/2,-1+1/2) {0};
      \node at (1-1/2,-2+1/2) {0};
      \node at (3-1/2,-1+1/2) {1};
      \node at (4-1/2,-1+1/2) {1};
      \node at (5-1/2,-1+1/2) {2};
      \node at (2-1/2,-2+1/2) {2};
      \node at (1-1/2,-3+1/2) {2};
      \node at (1-1/2,-4+1/2) {2};
      \node at (3-1/2,-2+1/2) {3};
      \node at (1-1/2,-5+1/2) {3};
      \node at (4-1/2,-2+1/2) {4};
      \node at (2-1/2,-3+1/2) {4};
      \node at (3-1/2,-3+1/2) {4};
      \node at (4-1/2,-3+1/2) {4};
      \node at (2-1/2,-4+1/2) {4};
    \end{scope}
  \end{tikzpicture}
  \caption{%
    The bijection between $\mathrm{LP}(\mathbb{L}(\lambda),r,c,n)$ and $\mathrm{RPP}(\lambda,n)$ where
    $\lambda = (5,4,4,2,1)$ with $r = 5$ rows and $c = 5$ columns, and $n = 4$.}
  \label{fig:bijection}
\end{figure}
It should be noted that this bijection is essentially the same as the classical interpretation of plane partitions by
``zig-zag'' non-intersecting paths \cite{Krattenthaler(2015arXiv)}.

We set up weight for reverse plane partitions which is equivalent to
the weight for lattice paths defined in Section \ref{sec:latticePaths}.
Let $\lambda' = (\lambda'_1,\dots,\lambda'_c)$ denote the partition conjugate with $\lambda$.
We define $\alpha_{i,j}$ by
\begin{align} \label{eq:alpha}
  \alpha_{i+k,\lambda_i+k} = a^{(r-i,c-\lambda_i)}_{n-k-1}, \qquad
  \alpha_{\lambda'_j+k,j+k-1} = b^{(r-\lambda'_j,c-j)}_{n-k}
\end{align}
for $1 \le i \le r$, $1 \le j \le c$ and $k < n$ where
$a^{(s,t)}_{n} \neq 0$, $b^{(s,t)}_{n} \neq 0$ is a solution to the discrete 2D Toda molecule \eqref{eq:d2DToda}.
We then define the weight of a reverse plane partition $\pi$ by
\begin{subequations} \label{eq:RPPWeight}
  \begin{align}
    v(\lambda,n;a,b;\pi) &= \prod_{(i,j) \in \lambda} v_{i,j}(\lambda,n;a,b;\pi) \qquad \text{with} \\
    \label{eq:weightParts}
    v_{i,j}(\lambda,n;a,b;\pi) &= \prod_{k=1}^{\pi_{i,j}} \frac{\alpha_{i+k-1,j+k-2}}{\alpha_{i+k-1,j+k-1}}.
  \end{align}
\end{subequations}

\begin{lem} \label{lem:weightBijection}
  Let $\lambda$ be a partition with $r$ rows and $c$ columns, and let $n \ge 0$.
  Assume that $\pi \in \mathrm{RPP}(\lambda,n)$ and $(P_0,\dots,P_{n-1}) \in \mathrm{LP}(\mathbb{L}(\lambda),r,c,n)$
  corresponds to each other by the bijection.
  Then
  \begin{align} \label{eq:weightBijection}
    v(\lambda,n;a,b;\pi) = \frac{\prod_{k=0}^{n-1} w(\mathbb{L}(\lambda);a,b;P_k)}
    {\prod_{i=1}^{r} \prod_{k=1}^{n} a^{(r-i,c-\lambda_i)}_{k-1}}.
  \end{align}
\end{lem}

\begin{proof}[Sketch of proof]
  Actually $\alpha_{i,j}$ is defined so that $v(\pi) = v(\lambda,n;a,b;\pi)$ is proportional to
  $\prod_{k=0}^{n-1} w(P_k)$ with $w(P) = w(\mathbb{L}(\lambda);a,b;P)$.
  That is, there exists a constant $\kappa$ such that $v(\pi) = \kappa \prod_{k=0}^{n-1} w(P_k)$.
  From \eqref{eq:RPPWeight}, $v(\lambda,n;a,b;\pi^{\emptyset}) = 1$ for
  the {\em empty} reverse plane partition $\pi^{\emptyset} \in \mathrm{RPP}(\lambda,n)$ whose parts are all 0.
  Thus $\kappa^{-1} = \prod_{k=0}^{n-1} w(P^{\emptyset}_k)$ where
  $(P^{\emptyset}_0,\dots,P^{\emptyset}_{n-1}) \in \mathrm{LP}(\mathbb{L}(\lambda),r,c,n)$ corresponds to
  $\pi^{\emptyset}$ by the bijection.
  We observe that $P^{\emptyset}_0$ goes from $(r,0)$ to $(0,c)$ along the border of $\mathbb{L}(\lambda)$, and
  $P^{\emptyset}_1,\dots,P^{\emptyset}_{n-1}$ are copies of $P^{\emptyset}_0$.
  Especially $w(P^{\emptyset}_k) = \prod_{i=1}^{r} a^{(r-i,c-\lambda_i)}_k$ and hence
  $\kappa^{-1}$ is equal to the denominator of the right-hand side of \eqref{eq:weightBijection}.
\end{proof}

The following is the main theorem of this paper.

\begin{thm} \label{thm:PFsProd}
  Let $a^{(s,t)}_{n} \neq 0$, $b^{(s,t)}_{n} \neq 0$ be a solution to the discrete 2D Toda molecule \eqref{eq:d2DToda}.
  Let $\lambda$ be a partition with $r$ rows and $c$ columns, and let $n \ge 0$.
  Then
  \begin{align}
    \label{eq:PFProdPartwise}
    \sum_{\pi \in \mathrm{RPP}(\lambda,n)} v(\lambda,n;a,b;\pi)
    = \prod_{i=1}^{r} \prod_{k=1}^{n} \frac{a^{(r-i,c)}_{k-1}}{a^{(r-i,c-\lambda_i)}_{k-1}}.
  \end{align}
\end{thm}

\begin{proof}
  This theorem is a translation of Theorem \ref{thm:NILatticePaths} via the bijection with the help of
  \eqref{eq:tauProd} and Lemma \ref{lem:weightBijection}.
\end{proof}

Theorem \ref{thm:PFsProd} allows us to find
a multiplicative partition function for reverse plane partitions of arbitrary shape with bounded parts from
each non-vanishing solution to the discrete 2D Toda molecule \eqref{eq:d2DToda}.

\section{An example}
\label{sec:example}

The discrete 2D Toda molecule \eqref{eq:d2DToda} has the solution
\begin{subequations} \label{eq:solution_apq}
  \begin{align}
    a^{(s,t)}_{n} &= [p]_{s+1}^{s+n} (1 - a [p]_{1}^{s} [q]_{1}^{t+n}), \\
    b^{(s,t)}_{n} &= a [p]_{1}^{s+n-1} [q]_{1}^{t} (1 - [q]_{t+1}^{t+n})
  \end{align}
\end{subequations}
with the notation that
$[z]_{m}^{n} = \prod_{\ell=m}^{n} z_{\ell}$ if $m \le n$,
$[z]_{m}^{n} = 1$ if $m = n+1$ and
$[z]_{m}^{n} = \prod_{\ell=n}^{m-1} z_{\ell}^{-1}$ if $m \ge n+2$.
The solution involves indeterminates $a$ and $p_{\ell}$, $q_{\ell}$ for $\ell \in \mathbb{Z}$ as parameters.

Let $\lambda$ be a partition with $r$ rows and $c$ columns.
Assume that
\begin{align}
  a = [x]_{c-\lambda'_{c}}^{\lambda_r-r}, \qquad
  p_i = [x]_{c-r+i-\mu_{i}}^{c-r+i-\mu_{i+1}}, \qquad
  q_j = [x]_{\mu'_{j+1}-j-r+c}^{\mu'_{j}-j-r+c}.
\end{align}
We then have a solution
\begin{subequations} \label{eq:specificSolution}
  \begin{align}
    a^{(s,t)}_{n} &= [x]_{c-r+s+1-\mu_{s+1}}^{c-r+s+n-\mu_{s+n+1}} (1 - [x]_{\mu'_{t+n+1}-t-n-r+c}^{c-r+s-\mu_{s+1}}), \\
    b^{(s,t)}_{n} &= [x]_{\mu'_{t+1}-t-r+c}^{c-r+s+n-1-\mu_{s+n}} (1 - [x]_{\mu'_{t+n+1}-t-n-r+c}^{\mu'_{t+1}-t-1-r+c}).
  \end{align}
\end{subequations}
Let $n \ge 0$.
The solution \eqref{eq:specificSolution} yields the weight of \eqref{eq:RPPWeight} with \eqref{eq:alpha} given by
\begin{align}
  \label{eq:specificWeight}
  v(\lambda,n;a,b;\pi) = \prod_{\ell=1-r}^{c-1} x_{\ell}^{\mathsf{tr}_{\ell}(\pi)}
  \prod_{(i,j) \in \lambda} \prod_{k=1}^{\pi_{i,j}} \frac{1 - [x]_{-n+j+k-1-\lambda'_{-n+j+k-1}}^{j-i-1}}{1 - [x]_{-n+j+k+\lambda'_{-n+j+k}}^{j-i}}.
\end{align}
As an instance of Theorem \ref{thm:PFsProd} we obtain
the following multiplicative partition function for reverse plane partition.

\begin{thm} \label{thm:partitionFunctionExample}
  Let $\lambda$ be a partition with $r$ rows and $c$ columns, and let $n \ge 0$.
  Then
  \begin{align} \label{eq:partitionFunctionExample}
    \sum_{\pi \in \mathrm{RPP}(\lambda,n)} v(\lambda,n;a,b;\pi)
    = \prod_{(i,j) \in \lambda} \frac{1 - [x]_{-n+j-\lambda'_{-n+j}}^{\lambda_i-i}}{1 - [x]_{j-\lambda'_j}^{\lambda_i-i}}
  \end{align}
  where the weight $v(\lambda,n;a,b;\pi)$ is given by \eqref{eq:specificWeight}.
\end{thm}

\begin{proof}
  Substituting the solution \eqref{eq:specificSolution} for the right-hand side of \eqref{eq:PFProdPartwise} we get
  \begin{subequations}
    \begin{align}
      \prod_{i=1}^{r} \prod_{k=1}^{n} \frac{1 - [x]_{-k+1-\lambda'_{-k+1}}^{\lambda_i-i}}{1 - [x]_{-k+1+\lambda_i-\lambda'_{-k+1+\lambda_i}}^{\lambda_i-i}}
      &= \prod_{i=1}^{r} \prod_{j=1}^{\lambda_i} \prod_{k=1}^{n} \frac{1 - [x]_{j-k-\lambda'_{j-k}}^{\lambda_i-i}}{1 - [x]_{j-k+1-\lambda'_{j-k+1}}^{\lambda_i-i}}
      \\
      &= \prod_{i=1}^{r} \prod_{j=1}^{\lambda_i} \frac{1 - [x]_{-n+j-\lambda'_{-n+j}}^{\lambda_i-i}}{1 - [x]_{j-\lambda'_{j}}^{\lambda_i-i}}.
    \end{align}
  \end{subequations}
  The last product is the same as the right-hand side of \eqref{eq:partitionFunctionExample}.
\end{proof}

The multiplicative partition function in Theorem \ref{thm:partitionFunctionExample} generalizes
the multi-trace generating function \eqref{eq:Gansner} by Gansner.
Indeed \eqref{eq:partitionFunctionExample} reduces into \eqref{eq:Gansner} as $n \to \infty$ because
$\lim_{n \to \infty} [x]_{-n+\text{const.}}^{\text{const.}} = 1$ as formal power series,
$\lim_{n \to \infty} \lambda'_{-n} = r$ and
$\lim_{n \to \infty} v(\lambda,n;a,b;\pi) = \prod_{\ell=1-r}^{c-1} x_{\ell}^{\mathsf{tr}_{\ell}(\pi)}$.

Assuming $x_{\ell} = q$ for $\ell \in \mathbb{Z}$ we obtain the partition function
\begin{subequations} \label{eq:weightExample_q}
  \begin{align}
    \sum_{\pi \in \mathrm{RPP}(\lambda,n)} v(\lambda,n;a,b;\pi)
    &= \prod_{(i,j) \in \lambda} \frac{1 - q^{\lambda_i+\lambda'_{j-n}-i-j+n+1}}{1 - q^{\lambda_i+\lambda'_j-i-j+1}}
    \qquad \text{with} \\
    v(\lambda,n;a,b;\pi)
    &= q^{|\pi|} \prod_{(i,j) \in \lambda} \prod_{k=1}^{\pi_{i,j}} \frac{1 - q^{n-i-k+1+\lambda'_{-n+j+k-1}}}{1 - q^{n-i-k+1+\lambda'_{-n+j+k}}}
  \end{align}
\end{subequations}
from \eqref{eq:specificWeight} and \eqref{eq:partitionFunctionExample}.
If $\lambda = (c^r)$, an $r \times c$ rectangular shape,
that becomes the triple product formula \eqref{eq:MacMahon} by MacMahon.
The partition function \eqref{eq:weightExample_q} is thus regarded as
a generalization of \eqref{eq:MacMahon} for reverse plane partitions of arbitrary shape.

\providecommand{\bysame}{\leavevmode\hbox to3em{\hrulefill}\thinspace}
\providecommand{\MR}{\relax\ifhmode\unskip\space\fi MR }
\providecommand{\MRhref}[2]{%
  \href{http://www.ams.org/mathscinet-getitem?mr=#1}{#2}
}
\providecommand{\href}[2]{#2}

\end{document}